\documentclass[reqno]{amsart}
\usepackage{amssymb, amsmath}
\usepackage{natbib}
\usepackage{lscape}
\usepackage{dsfont}
\usepackage{mathrsfs}
\usepackage{bbm}

\usepackage[pdftex,plainpages=false,colorlinks,hyperindex,bookmarksopen,linkcolor=red,citecolor=blue,urlcolor=blue]{hyperref}

\DeclareMathAlphabet{\mathpzc}{OT1}{pzc}{m}{it}

\bibpunct{[}{]}{;}{n}{,}{,}

\newtheorem{te}{Theorem}[section]
\newtheorem{defin}[te]{Definition}
\newtheorem{os}[te]{Remark}
\newtheorem{prop}[te]{Proposition}

\newtheorem{ex}[te]{Example}

\numberwithin{equation}{section}

\allowdisplaybreaks

\def \l { \left( }
\def \r {\right) }
\def \ll { \left\lbrace }
\def \rr { \right\rbrace }

\begin{document}

	\title[]{Semi-Markov models and
	 motion in heterogeneous media}
	\address{$1$: Dipartimento di Scienze Statistiche, Sapienza - Universit\`a di Roma}
	\address{$2$: Dipartimento di Matematica e applicazioni ``Renato Caccioppoli" - Universit\`a degli studi di Napoli ``Federico II"}
	\keywords{Semi-Markov processes, anomalous diffusion, continuous time random walks, Volterra equations, fractional derivatives, subordinators}
	\date{\today}
	\subjclass[2010]{60K15, 60K40, 60G22}

\author{Costantino Ricciuti$^1$}
\author{Bruno Toaldo$^2$}
		\begin{abstract}
In this paper we study continuous time random walks (CTRWs) such that the holding time in each state has a distribution depending on the state itself. For such processes, we provide integro-differential (backward and forward) equations of Volterra type, exhibiting a position dependent convolution kernel. Particular attention is devoted to the case where the holding times have a power-law decaying density, whose exponent depends on the state itself, which leads to variable order fractional equations. A suitable limit yields a variable order fractional heat equation, which models anomalous diffusions in heterogeneous media. 
\end{abstract}

	\maketitle
\tableofcontents


\section{Introduction}
We here consider continuous time random walks (CTRWs) on countable state spaces. It is assumed that every time the walker jumps, the future trajectory becomes independent of its past, namely the next position and the next jump time  
depend only on the current position; furthermore, in a generic time instant, the future behavior is assumed to be also depending  on the time already spent in the current position. Such a process is said to be semi-Markovian.
If the waiting times between jumps follow an exponential distribution, then, due to the lack of memory property, the random walk is a Markov process.

It is well known that suitable (Markovian) random walks are good approximations of the Brownian motion. In the last decades it has been noticed that the CTRWs whose waiting times  have distribution  with a power-law decay, played a central role in statistical physics because they are good approximations of anomalous diffusion processes, where the mean square displacement grows as $\overline{x^2} \sim t^{\alpha}, \alpha \in (0,1)$, and therefore slower than a standard Brownian motion (for a complete overview on this matter consult \cite{Metzler} and references therein). In these models each site exercises a trapping effect which, in some sense, delayes the time with respect to a corresponding Markov process.

It turns out that these facts can be framed in a nice probabilistic setting: to construct a large class of CTRWs, it is sufficient to replace the deterministic time $t$ of a Markov process by an independent inverse stable subordinator (on this point see, for example, the instructive discussion in \cite{Meerschaert1}).
It is well known (see for example  \cite[page 365]{Kolokoltsov} and \cite{Meerschaert3}) that the transition probabilities of the correponding CTRW follows both the fractional backward and forward equations. Such equations are obtained from Kolmogorov backward and forward equations by replacing the time derivative with the fractional one, which introduces a memory effect by a convolution integral with a slowly decaying power-law kernel.

A suitable scaling then leads to anomalous diffusion processes, whose p.d.f. solves the  Fokker-Planck equation (see, e.g., \cite{Metzler})
\begin{align}
\frac{\partial}{\partial t} p(x,y,t)= k \mathcal{D}_t^{1-\alpha}\frac{\partial^2}{\partial x^2} p(x,y,t).
\end{align}
It has been empirically confirmed (see \cite{Metzler}) that these models are particularly effective in a number of applications, e.g., for modeling diffusion in percolative and porous systems, charge carrier transport in amorphous semiconductors, nuclear magnetic resonance, motion on fractal
geometries, dynamics of a bead in a polymeric network, protein conformational
dynamics and many others. 


We finally stress that many aspects of the theory hold as they are if the distribution of the holding times is arbitrary and not necessarily with a power law decay provided that it satisfies some mild assumptions (see, for example, the discussion in \cite[Section 4]{Meerschaert1}). In this case the random time process is given by the inverse of a generic subordinator and the corresponding backward equations have the form of a Volterra integro-differential equation
\begin{align}
\frac{d}{dt} \int_0^t p(x,y,s) \, k(t-s) \, ds \, - \, k(t) p(x,y,0) \, = \, \sum_{z} g_{x,z} p(z,y,t),
\label{voltintro}
\end{align}
where $g_{x,y}:=(G)_{x,y}$ and $G$ is the Markovian generator (see \cite{Zhen Qing Chen, Meerschaert3, toaldo, toaldodo} for the general theory and \cite{dovetalip, dovetal} for some particular cases).
The fractional case is more familiar in statistical physics because it is widely used in applications.

Up to now, we have only considered 
the simplest forms of fractional kinetic equations, where the fractional index $\alpha$ is constant. On the other hand, it is clear that further theoretical investigations are required  for the description of more complicated (and more realistic) random processes, where the particle moves in an inhomogeneous environment; as we will discuss in the paper it turns out that this leads to equations of multi-fractional type. Equations with time-fractional derivative whose order depends on space have been studied in \cite{Orsingher2}, where the authors considered a CTRW, say $X(t)$, $t \geq 0$, such that  the function $ f(x,t)=\mathds{E}\{u(X(t))\mid X(0)=x \}$, for  a suitable test function $u$, solves the fractional backward equation
\begin{align}
\mathcal{D}_t^{\alpha (x)}f(x,t)=Gf(x,t) \label{intro1}
\end{align}
where $\mathcal{D}_t^{\alpha (x)}$ denotes the $\alpha$-fractional derivative in the sense of Caputo-Dzerbayshan, i.e., for $\alpha \in (0,1)$,
\begin{align}
\mathcal{D}_t^{\alpha } u(t) \, : = \, \frac{1}{\Gamma(1-\alpha)} \frac{d}{dt} \int_0^t   u(s) \, (t-s)^{-\alpha} \, ds \, - \, \frac{t^{-\alpha}u(0)}{\Gamma(1-\alpha)},
\label{defcaputo}
\end{align}
for any function $u$ such that the above integral is differentiable. 
In such a model, the trapping effect is not exercised with the same intensity at all sites.
Indeed, when the particle reaches the state $x$, it is trapped for a time interval with density
 $\psi (t)\sim t^{-1-\alpha (x)}$ before jumping to another point.
Thus the time  delay is stronger when the particle is located at points with small values of $\alpha$.
This leads to the fundamental fact regarding the time-change relation $X(t)=M(L(t))$: the time process $L$  and the Markov process $M$ are not independent. Such a construction is far from trivial, since $L$ is the right continuous inverse of a non-decreasing additive process also called time-inhomogeneous subordinator (for basic information consult \cite{Sato} and \cite{Orsingher1}).

In the case of a countable state space $\mathcal{S}$, we here present the derivation of the backward equation 
\begin{align}
\mathcal{D}_t^{\alpha (x)}p(x,y,t)= \sum _z g_{x,z}p(z,y,t).
\end{align}
We further introduce the forward equation
\begin{align}
\frac{d}{dt} p(x,y,t)= \sum _zg_{z,y}\, ^R \mathcal{D}_t^{1-\alpha (z)} p(x,z,t)
\label{forwgiusta}
\end{align}
where $^R\mathcal{D}_t^{1-\alpha (x)}$ denotes the fractional derivative in the sense of Riemann-Liouville, i.e., for $\beta \in (0,1),$
\begin{align}
^R\mathcal{D}_t^\beta u(t) : = \frac{1}{\Gamma(1-\beta)} \frac{d}{dt} \int_0^t u(s) \, (t-s)^{-\beta} ds
\label{defriemann}
\end{align}
for any function $u$ such that the above operator is well defined.
Further we explain why eq. \eqref{forwgiusta} is a true forward equation in the classical sense  of Kolmogorov.

Therefore, this paper also creates a further bridge between the theory of semi-Markov processes and models of motion in heterogeneous media (a different theory concerning motions at finite velocity is discussed in \cite{koro}): on the one hand, there is the theory of semi-Markov processes, on the other hand, there are recent works  concerning the fractional diffusion equation with multifractional index:
\begin{align}
\frac{\partial}{\partial t}p(x,y,t)=\frac{1}{2} \frac{\partial ^2}{\partial x^2}\l k(x) \mathcal{D}_t^{1-\alpha (x)}p(x,y,t) \r.
\end{align}
Such equation has been derived in \cite{Gorenflo},
but the related theory  is still at an early stage, especially with regard to physical and phenomenological aspects. However, the use of a multi-fractional index $\alpha (x)$ is more realistic in the description of physical phenomena. Indeed, it takes into account the possibility of heterogeneous media, or, more simply, it considers homogeneous media where some impurities are scattered.

Finally, in the same spirit as \eqref{voltintro} we show that previous models can be generalized by letting the (random) trapping effects having an arbitrary density, subject to some mild assumptions. These models then yield integro-differential equations of Volterra type, with a position dependent kernel of convolution, i.e.,
\begin{align}
\frac{d}{dt} p(x,y,t)= \sum _zg_{z,y}\, \frac{d}{dt} \int_0^t p(x,z,s) \, k(t-s,z) \, ds.
\end{align}

The plan of the paper is the following.  In sections \ref{2}, \ref{3}, \ref{4} and \ref{5}, we consider the case of power-law holding times, which is the most familiar case in statistical physics. In particular, in section 2 we review (in our notations) some known facts on CTRWs in a homogeneous environment, where the fractional index $\alpha$ is assumed to be constant in space.  Section 3 and 4 regard CTRWs in heterogeneous environment, where the fractional index is assumed to be space-dependent. Section 5 deals with the derivation of the multifractional diffusion equation. In section 6 many results are extended to the case where the holding times follow more general distributions.

\section{Semi-Markov models for motion in homogeneous media}
\label{2}
Before moving to heterogenous media we collect some results from the literature concerning classical models which will be used in the subsequent parts. As we stated in the introduction, the most popular model in statistical physics is related with holding times in each site having a density $\psi(t) \sim Ct^{-\alpha-1}$, $C>0$, $\alpha \in (0,1)$, with a power law decay. So, for example $\psi (t) = -(d/dt) E_{\alpha}(-\lambda t^{\alpha})$ (compare with \cite[eq. (26)]{Scalas2006Lecture}), and this is related to fractional processes. Hence we focus the attention on this case to present the results concerning this theory.

First, in order to introduce the notation that we will use hereafter, we recall some basic facts regarding the classical theory of stepped Markov processes.
Let us consider  a continuous time  Markov process $M$ with discrete state space $\mathcal{S}$ 
\begin{align} \label{definizione processo markov}
M(t)=X_n \qquad  V_n  \leq t< V_{n+1}\qquad \text{ where } V_0=0 \qquad V_n= \sum _{k=0}^{n-1} E_k
\end{align}
where $X_n$ is a homogeneous discrete-time Markov chain on $\mathcal{S}$ with transition probabilities
\begin{align}
h_{i,j}=P(X_{n+1}=j|X_n=i), \qquad \forall n \in \mathbb{N} \qquad i,j \in \mathcal{S},
\end{align} 
and the sojourn times are such that
\begin{align}
P(E_n>t|X_n=i)=e^{-\lambda_i t} \qquad \forall n\in \mathbb{N}, \qquad t\geq 0.
\end{align}
Let 
\begin{align}
p_{i,j}(t)= P(M(t)=j|M(0)=i)
\end{align}
be the transition probabilities. 
The Markovian generator of $M$ is the matrix with elements
\begin{align}
g_{i,j}= \lambda _i (h_{i,j}-\delta _{i,j})
\end{align}
where $\delta _{i,j}$ denotes the Kronecker symbol.
Then the infinitesimal transition probabilities have the form
\begin{align}
p_{i,j}(dt)=\begin{cases} g_{i,j} dt = \lambda _i h_{i,j}dt,  &i\neq j,\\
1+g_{i,i} dt= 1-\lambda _i dt +\lambda _i h_{i,i}dt, \qquad & i=j. \end{cases}
\end{align}
It is enough for our models to consider the case in which, a.s., 
\begin{align}
\zeta := \sup_n V_n = \sum_n E_n = \infty, 
\end{align}
so that the processes here are non explosive and hence we shall not consider  what happens to a process after explosion. Under all these assumptions the functions $p_{i,j}(t)$, with  $i,j \in S$ solve the Kolmogorov backward equations (e.g. \cite[Sec. 2.8]{norris})
\begin{align} \label{backward markoviana}
\frac{d}{dt} p_{i,j}(t)= \sum _k g_{i,k}p_{k,j}(t), \qquad p_{i,j}(0)= \delta _{i,j},
\end{align}
as well as the Kolmogorov forward equations
\begin{align} \label{forward markoviana}
\frac{d}{dt} p_{i,j}(t)= \sum _k p_{i,k}(t) g_{k,j}, \qquad p_{i,j}(0)= \delta _{i,j},
\end{align}
which can be written in compact matrix notation as
\begin{align}
\frac{d}{dt}P(t)= GP(t)=P(t)G, \qquad P(0)=I. 
\end{align}

We now consider a CTRW constructed in the same way of $M$, except for the distribution of the waiting times, which are no longer exponentially distributed. These processes are said to be semi-Markov processes in the sense of Gihman and Skorohod \cite[Chapter 3]{gihman}. Hence let $X(t)$ be
\begin{align}
X(t)=X_n ,\qquad  T_n  \leq t< T_{n+1},\qquad \text{where } T_0=0  \qquad T_n= \sum _{k=0}^{n-1} J_k,\label{processo principale}
\end{align} 
where $X_n$ is a homogeneous discrete time Markov chain on $\mathcal{S}$ with transition probabilities
\begin{align}
h_{i,j}=P(X_{n+1}=j|X_n=i) ,\qquad \forall n \in \mathbb{N}, \qquad i,j \in \mathcal{S},
\end{align} 
and the sojourn times are such that
\begin{align}
P(J_n>t|X_n=i)=\overline{F}_i(t), \qquad \forall n\in \mathbb{N} ,\qquad t\geq 0,
\end{align}
where $F_i(t) = 1-\overline{F}_i(t)$ is an arbirtrary c.d.f. We will devote particular attention to the case
\begin{align}
P(J_n>t|X_n=i)=E_\alpha(-\lambda _i t^{\alpha }), \qquad \forall n\in \mathbb{N} ,\qquad t\geq 0,
\end{align}
for $\lambda_i>0$, where 
\begin{align*}
E_\alpha(x):= \sum_{k=0}^\infty \frac{x^k}{\Gamma (1+\alpha k)}
\end{align*}
is the Mittag-Leffler function. In this case (e.g. \cite[eq. (26)]{Scalas2006Lecture})
\begin{align}
E_{\alpha}(- \lambda t^{\alpha }) \sim C_\lambda \frac{t^{-\alpha}}{\Gamma(1-\alpha)}, \qquad C_\lambda >0,
\end{align}
and the corresponding equations are fractional. The charactering property of semi-Markov processes is the following: by defining
\begin{align*}
\gamma(s) = s-\sup \ll w \leq s : X(w) \neq X(s) \rr ,
 \end{align*}
the sojourn time of $X$ in the current position,  the couple $(X(t), \gamma (t))$  is a (strict) Markov process \cite[Chapter 3, Section 3, Lemma 2]{gihman}. This is to say that, when conditioning on the trajectory up to time $s$, future events depend not only on the current position $X(s)$, as it is for Markov processes, but also on the quantity $\gamma(s)$.  
Let 
\begin{align} \label{bbbb}
p_{i,j}(t)&:=P(X(t)=j|X(0)=i, \gamma(0) =0) \notag \\& =P(X(t+\tau)=j|X(\tau)=i, \gamma(\tau) =0) 
\end{align}
be the transition probabilities (the second equality follows by time-homogeneity). 
We know from \cite[page 20]{koro} that the transition probabilities solve the renewal equation
\begin{align} \label{Markov renewal equation}
p_{i,j}(t)= P \{ J_i>t\}\delta _{i,j}+ \int _0^t \sum _l h_{i,l}\, p_{l,j}(t-s)\mathpzc{f}_i(ds) 
\end{align}
where here $\mathpzc{f}_i(t)$ denotes a smooth density of $F_i(t)$.
Note that (\ref{Markov renewal equation}), which provides a system of integral equations for the transition probabilities \eqref{bbbb}, comes from a very classical conditioning argument: 
fixing the time of the first jump $J_0$ and using the Markov property of the semi-Markov process at the jump times  yields  (see \cite[page 19]{koro} for some details)
\begin{align}
P \l X(t) = j \mid X(0) = i, \gamma(0)=0 \r \, = \, &P \l X(t) = j, J_0 >t \mid X(0) = i , \gamma(0)=0 \r \notag \\ &+ P \l X(t) = j , J_0 \leq t \mid X(0) = i, \gamma(0)=0 \r.
\end{align}
A similar approach on semi-Markov processes, with an interesting discussion on exactly solvable models, can be found in \cite{Scalas}.

The process \eqref{processo principale} is known to have a deep connection to fractional calculus. Indeed, the following result holds.
\begin{prop}
\label{teofrachom}
 The transition functions $p_{i,j}(t)$, $i,j \in \mathcal{S}$,  defined in \eqref{bbbb}
solve the following system of backward equations
\begin{align} \label{aaaa}
\mathcal{D}_t^{\alpha } \, p_{i,j}(t)&= \sum _k g_{i,k}p_{k,j}(t) , \qquad p_{i,j}(0)= \delta _{i,j},
\end{align}
as well as the system of ``forward" equations
\begin{align} \label{zzz}
\mathcal{D}_t^{\alpha } \, p_{i,j}(t)&= \sum _k p_{i,k}(t) g_{k,j}, \qquad p_{i,j}(0)= \delta _{i,j} .
\end{align}
\end{prop}
\begin{proof}
By the convolution Theorem we can compute the Laplace transform in \eqref{aaaa} and we obtain
\begin{align} 
s^{\alpha}\widetilde{p}_{i,j}(s)-\frac{s^{\alpha}}{s}p_{i,j}(0) \, =  \, \sum _k g_{i,k}\widetilde{p}_{k,j}(s). \label{cccc}
\end{align}
Instead, by applying the Laplace tranform to (\ref{Markov renewal equation}) we have
\begin{align} \label{secondac}
\widetilde{p}_{i,j}(s)= \frac{s^{\alpha} }{s(\lambda _i+s^{\alpha} )}\delta _{i,j} +\sum _l h_{i,l}\, \widetilde{p}_{l,j}(s)\frac{\lambda _i}{\lambda _i+ s^{\alpha}}.
\end{align}
By setting $g_{i,j}=\lambda _i (h_{i,j}-\delta _{i,j})$, it is easy to show that \eqref{secondac} reduces to \eqref{cccc} and using the uniqueness theorem for Laplace tranforms eq. \eqref{aaaa} is proved. Now if we apply again the Laplace transform to \eqref{zzz} we get that
\begin{align}
s^{\alpha}\widetilde{p}_{i,j}(s)- \frac{s^{\alpha }}{s}p_{i,j}(0)= \sum _k g_{k,j}\widetilde{p}_{i,k}(s), \label{ccccC}
\end{align}
and thus the solution of \eqref{aaaa} and \eqref{zzz} coincide. Indeed they can be obtained by solving either the system  \eqref{cccc} or \eqref{ccccC} in the variables $\tilde{p}_{i,j}(s)$, which, in compact operator form, reads
\begin{align} \label{soluz 1}
\tilde{P}(s)= s^{\alpha -1}(s^{\alpha}I-G)^{-1}
\end{align}
where $I$ is the identity matrix, and this concludes the proof.
\end{proof}
\begin{os}
We observe that \eqref{aaaa} derives directly by the renewal equation, which has a clear backward meaning. Instead the reason why we call eq. \eqref{zzz} ``forward equation" is that it is formally obtained by introducing the fractional Caputo derivative in the Kolmogorov forward equation \eqref{forward markoviana}; the fact that it has a clear probabilistic interpretation has never been proved.
A clear probabilistic meaning to \eqref{zzz} will be derived later in Section \ref{4} from the general form of the forward equation of semi-Markov processes we will present. Concerning forward equations of semi-Markov processes, see also the discussion in \cite{fellersemi}.
\end{os}

It is well known that \eqref{processo principale} can be equivalently constructed by replacing the time $t$ in \eqref{definizione processo markov} with the right continuous inverse of an independent $\alpha$-stable subordinator. For the sake of clarity we here report a sketched proof of this  fact, which essentially follows \cite[Theorem 2.2]{Meerschaert1}. Let $H$ and $L$ respectively denote the $\alpha$-stable subordinator and its inverse, i.e.,
\begin{align}
L(t) \, := \, \inf \ll s \geq 0 : H(s) >t \rr.
\end{align}
In order to prove that \eqref{processo principale} is the same process as $M(L(t))$ it is sufficient to prove that $M(L(t))$ has the same Mittag-Leffler intertimes of \eqref{processo principale}. This is clear since to construct a semi-Markov process in the sense of Gihman and Skoroohod, as in Section \ref{2}, it is sufficient to have an embedded chain $X_n$ and a sequence of independent r.v.'s representing the holding times. Here $M(t)$ and \eqref{processo principale} have the same embedded chain and thus it remains only to show that they have the same waiting times. Since
\begin{align}
M(t)=X_n \qquad  V_n  \leq t< V_{n+1}
\end{align}
we have
\begin{align}
M(L (t))=X_n \qquad  V_n  \leq L(t) < V_{n+1}
\end{align}
which is equivalent to (by \cite[Lemma 2.1]{Meerschaert1})
\begin{align}
M(L (t))=X_n \qquad H(V_n-)  \leq t < H(V_{n+1}-).
\label{limsin}
\end{align}
Further, by \cite[Lemma 2.3.2]{applebaum} we have that $H(V_n-)=H(V_n)$, a.s., we can rewrite \eqref{limsin} as
\begin{align}
M(L (t))=X_n \qquad H(V_n)  \leq t < H(V_{n+1}).
\end{align}
Thus the jump times $\tau_n$  of $M(L(t))$  are such that $\tau_n \stackrel{d}{=} H(V_n)$ and since $H$ has stationary increments, the holding times of $M(L(t))$ become, for any $n$
\begin{align}
 \tau_{n+1}-\tau_n= H (V_{n+1})-H (V_n) \stackrel{d}{=} H (E_n),
 \end{align}
where we used that $V_{n+1}-V_n$ are exponential r.v.'s $E_n$.
By a standard conditioning argument we have, under $P \l \cdot \mid X_n=x \r$,
\begin{align}
\mathds{E}e^{-\eta H(E_n)}= \frac{\lambda _x}{\lambda _x +\eta ^{\alpha}}.
\end{align}
Now since \cite[eq. (3.4)]{meerbounded}
\begin{align}
\int_0^\infty e^{-\eta t} E_{\alpha}(-\lambda_x t^{\alpha})dt \, = \, \eta^{\alpha-1} \frac{1}{\lambda_x + \eta^{\alpha}}
\end{align}
we have,
\begin{align}
-\int_0^\infty e^{-\eta t} \frac{d}{dt} E_{\alpha}(-\lambda_x t^{\alpha})dt \, = \, \frac{\lambda_x}{\lambda_x+\eta^{\alpha}}
\label{lapldensmittag}
\end{align}
and this shows that the holding times have the same distribution.

\begin{ex}[The fractional Poisson process]
One of the most popular CTRW with heavy tailed waiting times is the so-called Fractional Poisson process, corresponding to the case where $\lambda_i= \lambda$, $h_{i,i+1}=1$ and $X(0)=0$ a.s. It has been studied by a number of authors (see for example \cite{Beghin2, LaskinS, MainardiS, Meerschaert1, RepinS}).

Its transition probabilities  \eqref{bbbb}
 solve the system of fractional Kolmogorov ``forward" equations
\begin{align}
\mathcal{D}_t ^\alpha p_{i,j}(t)& = -\lambda p_{i,j}(t)+\lambda p_{i,j-1}(t) \qquad j\geq i \qquad p_{i,j}(0)=\delta _{i,j}
\end{align}
as well as the system of Kolmogorov fractional backward equations
\begin{align}
\mathcal{D}_t^\alpha p_{i,j} (t)&= -\lambda p_{i,j}(t)+\lambda p_{i+1,j}(t) \qquad j\geq i \qquad p_{i,j}(0)= \delta _{i,j}
\end{align}
and it is easy to check directly that their common explicit solution in Laplace space is
\begin{align}
\widetilde{p}_{i,j}(s)= \frac{\lambda ^{j-i}s^{\alpha -1}}{(\lambda +s^{\alpha})^{j-i+1}}\label{soluzione poissson classico}.
\end{align}
However, the equation often reported in the literature (e.g. \cite{Beghin}) is 
\begin{align}
\mathcal{D}_t^\alpha p_k (t)&= -\lambda p_k(t)+\lambda p_{k-1}(t) \qquad k\geq 0 \label{equazione di poisson frazionaria} \\ p_k(0)&= \delta _{k,0} 
\notag
\end{align}
which is the ``forward" equation corresponding to the special case $i=0$. In \cite{Meerschaert1}, the authors proved that the fractional Poisson process can be constructed as a standard Poisson process with the time variable replaced by an inverse stable subordinator. 
\end{ex}

\section{Semi-Markov models for motion in heterogeneous media}
\label{3}
We now show how the tools used for modeling homogeneous media can be adapted to include heterogeneity, in the sense that the trapping effect exercised in different sites depends on the current position. To be consistent  with the literature introduced in Section \ref{2} we first focus on the case in which the holding time at the position $x$ has a density $\psi(t) \sim t^{-\alpha(x)-1}$. How this can be generalized to different decaying patterns will be showed later. Hence we consider now a CTRW defined exactly as in \eqref{processo principale}, except for the distribution of the waiting times, which here present a state dependent fractional order:
\begin{align}
X(t)=X_n, \qquad   T_n \leq t < T_{n+1}, \qquad \text{where } T_0 =0 \qquad T_n= \sum _{k=0}^{n-1} J_k \notag \\
P(J_n>t|X_n=i)= \overline{F}_i(t)=E_{\alpha _i}(-\lambda _i t^{\alpha _i}) \qquad \alpha _i \in (0,1).
\label{processo da studiare}
\end{align}
Use again \cite[eq. (26)]{Scalas2006Lecture} to say that, for a constant $C>0$ depending on $\alpha_i$ and $\lambda_i$, we have, as $t \to \infty$
\begin{align}
-\frac{d}{dt} E_{\alpha _i}(-\lambda _i t^{\alpha _i}) \, \sim  \, C t^{-\alpha_i-1}
\end{align}
and thus this is a model of a motion performed in a medium where the trapping effect has not the same intensity at all sites.

Before moving to the equation, it is usefull to show that also in this situation it is possible to interpret the semi-Markov process $X(t)$ as the time-change of a Markov process. However this is far from trivial and requires some analysis which is carried out in the following section.

\subsection{The time-change by a dependent time process}
In order to have an interpretation of \eqref{processo da studiare}  as a time-changed process, we need the notion of multistable subordinator  (see for example \cite{molcha} and  \cite{Orsingher1}). Strictly speaking,  a multistable subordinator is a generalization of a stable subordinator, in the sense that the stability index is a function of time $\alpha= \alpha (t) \in (0,1)$. The intensity of jumps is described by a time-dependent L\'evy measure
\begin{align}
\nu (dx,t)= \frac{\alpha (t) x^{-\alpha (t)-1}dx}{\Gamma (1-\alpha (t))} \qquad x>0.
\label{levmulti}
\end{align}
A multistable subordinator $\sigma(t)$, $t \geq 0$ is an additive process in the sense of \cite{Sato}, i.e., it is right-continuous and has independent but non stationary increments. Hence all the finite dimensional distributions are completely determined by the distribution of the increments which can be obtained from \eqref{levmulti}. Therefore (see \cite[Section 2]{Orsingher1} for details on this point)
\begin{align}
\mathds{E}e^{-\eta (\sigma(t)-\sigma(s))}= e^{-\int _s ^t \eta ^{\alpha (\tau)}d\tau}, \qquad 0\leq s \leq t.
\label{laplincr}
\end{align}
Multistable subordinators are particular cases of a larger class of processes, known as non-homogeneous subordinators, which were introduced in \cite{Orsingher1}. 

\begin{defin}
\label{defpiec}
A multistable subordinator $\sigma(t)$, $t \geq 0$, is said to be piecewise stable if there exists a sequence  $\alpha _j \in (0,1)$ and a sequence $t_j\geq 0$ such that
 the stability index can be written as 
\begin{align}
\alpha (t)= \alpha _j \qquad    t_{j}  \leq t< t_{j+1}
\end{align}
and thus the time-dependent L\'evy measure has the form
\begin{align} \label{misura di levy multistabile a tratti}
\nu (dx,t) =  \frac{\alpha _j x^{-\alpha _j-1}}{\Gamma (1-\alpha _j)}dx , \qquad  t_j \leq t < t_{j+1}.
\end{align}
\end{defin}
Note that to each $\alpha _j \in (0,1)$ there corresponds a stable subordinator $H_{\alpha _j}(t)$ with index $\alpha _j$ in such a way that $\sigma $ is defined as
\begin{align} \label{incrementi multistabile a tratti}
\sigma (t)=\sigma (t_j)+ H_{\alpha _j} (t-t_j) \qquad \forall t\in [t_j, t_ {j+1}).
\end{align}
The following theorem shows that \eqref{processo da studiare} is given by a Markov process time-changed by the inverse of a piecewice stable subordinator.  The major novelty consists in the fact that the original process and the random time are not independent as in the classical case. This reflects the fact that the intensity of the trapping effect is not space homogeneous, i.e., the time delay depends on the current position.  

\begin{te}
\label{tetimechfrac}
Let $M$ be a Markov process defined as in (\ref{definizione processo markov}). Moreover, let $\sigma^M(t)$ be a multistable (piecewise stable) subordinator dependent on $M$ whose L\'evy measure is given, conditionally on $V_1 =v_1,V_2=v_2, \cdots$ and $X_1=x_1, X_2=x_2, \cdots$ by
\begin{align} 
\nu ^M(dx,t) =  \frac{\alpha _{x_j} x^{-\alpha _{x_j}-1}}{\Gamma (1-\alpha _{x_j})}dx,  \qquad     v_{j} \leq t< v_{j+1}.
\end{align}
Let $L^M(t)$ be the right-continuous process
\begin{align}
L^M(t) : = \, \inf \ll s \geq 0 : \sigma^M (s) > t \rr.
\end{align}
Then the time-changed process $M(L^M(t))$ is the same process as \eqref{processo da studiare}.
\end{te}
\begin{proof}
The proof is on the line of the discussion at the end Section \ref{2}. To prove that \eqref{processo da studiare} coincides with $M(L^M(t))$ it is sufficient to prove that $M(L^M(t))$ has the same Mittag-Leffler intertimes of \eqref{processo da studiare} since $M(t)$ and $X(t)$ have the same embedded chain.  Let $V_n$, $n\geq 1$, be the jump times of $M$. Since
\begin{align}
M(t)=X_n \qquad  V_n  \leq t< V_{n+1}
\label{this}
\end{align}
we have
\begin{align}
M(L^M (t))=X_n \qquad  V_n  \leq L^M(t) < V_{n+1}.
\end{align}
Now by \cite[Theorem 2.2]{Orsingher1} we know that $\sigma^M(t)$ is strictly increasing and then we can apply \cite[Lemma 2.1]{Meerschaert1} to say that \eqref{this} is equivalent to
\begin{align}
M(L^M (t))=X_n \qquad \sigma^M(V_n-)  \leq t < \sigma^M(V_{n+1}-).
\label{equiv}
\end{align}
Now use \cite[Theorem 2.1]{Orsingher1} to say that, a.s.,  $\sigma^M(t) = \sigma^M (t-)$ and thus to say that \eqref{equiv} is equivalent to
\begin{align}
M(L^M (t))=X_n \qquad \sigma^M(V_n)  \leq t < \sigma^M(V_{n+1}).
\end{align}
Thus the jump times $\tau_n$  of $M(L^M(t))$  are such that $\tau_n \stackrel{d}{=}\sigma ^M(V_n)$ and, by \eqref{incrementi multistabile a tratti}, the holding times are such that, under $P \l \cdot \mid X_n = x \r$,
\begin{align}
 \tau_{n+1}-\tau_n= \sigma^M (V_{n+1})-\sigma^M (V_n) \stackrel{d}{=} H_{\alpha _x} (E_n).
 \end{align}
By a standard conditioning argument we have
\begin{align}
\mathds{E}e^{-\eta H_{\alpha _x}(E_n)}= \frac{\lambda _x}{\lambda _x +\eta ^{\alpha _x}}.
\end{align}
This fact together with formula \eqref{lapldensmittag} concludes the proof.
\end{proof}

\subsection{Variable order backward equations}
We here derive the backward equation for the semi-Markov process \eqref{processo da studiare} in this new heterogeneous framework and we show that this equation becomes fractional of order $\alpha(i)$ where $i$ is the state where the transition is started.
\begin{te}
\label{teoback}
The transition functions of \eqref{processo da studiare} $p_{i,j}(t)$, $i,j \in \mathcal{S}$,  
solve the following system of backward equations
\begin{align} \label{aa}
\mathcal{D}_t^{\alpha _i} \, p_{i,j}(t)&= \sum _k g_{i,k}p_{k,j}(t) , \qquad p_{i,j}(0)= \delta _{i,j} .
\end{align}
\end{te}
\begin{proof}
We  can perform Laplace transform computation similar to that in the proof of Proposition \ref{teofrachom}. By applying the Laplace transform to \eqref{aa} we obtain
\begin{align} 
s^{\alpha _i}\widetilde{p}_{i,j}(s)-s^{\alpha _i-1}p_{i,j}(0)= \sum _k g_{i,k}\widetilde{p}_{k,j}(s). \label{ccc}
\end{align}
Instead, by applying the Laplace tranform to (\ref{Markov renewal equation}) we have
\begin{align} \label{seconda}
\widetilde{p}_{i,j}(s)= \frac{s^{\alpha _i} }{s(\lambda _i+s^{\alpha _i} )}\delta _{i,j} +\sum _l h_{i,l}\, \widetilde{p}_{l,j}(s)\frac{\lambda _i}{\lambda _i+ s^{\alpha _i}}.
\end{align}
By setting $g_{i,j}=\lambda _i (h_{i,j}-\delta _{i,j})$, it is easy to show that \eqref{seconda} can be rewritten as \eqref{ccc} and by the uniqueness theorem for Laplace tranform the desired result is immediate.
\end{proof}
The explicit form of the transition probabilities is easy obtained in Laplace space.
By applying the Laplace transform, the system of fractional equations \eqref{aa} reduces to the system of linear equations \eqref{ccc} in the variables $\widetilde{p}_{i,j}(s)$.
In compact matrix form, \eqref{ccc} can be written as
\begin{align}
\Lambda \widetilde{P}(s)- s ^{-1}\Lambda  I= G \widetilde{P}(s)
\end{align}
where $(\widetilde{P}(s))_{i,j}= \widetilde{p}(s)_{i,j}$, $I$ is the identity matrix, while
\begin{align}
\Lambda  = \text{diag} (s^{\alpha_1}, s ^{\alpha _2},\dots, s ^{\alpha _n}...).
\end{align}
Thus the solution in matrix form is written as
\begin{align}
\widetilde{P}(s)= \frac{1}{s}(\Lambda -G)^{-1} \Lambda I.
\end{align}

\section{The forward equations of semi-Markov processes in heterogeneous media}
\label{4}
In the spirit of what happens in the homogeneous case one can be tempted to look for the forward equation by trying to replace the ordinary time-derivative in \eqref{forward markoviana} with a variable-order Caputo derivative $\mathcal{D}_t^{\alpha (\cdot)}$, where $(\cdot)$ denotes the final state $j$. However, such an attempt is unsuccessful since it can be shown that the solution in the Laplace space does not coincide with the one of \eqref{aa}. We discuss here an example from the literature in which it is showed that this approach fails.

\begin{ex}[The state dependent fractional Poisson process]
\label{expoi1}
In the pioneering work \cite{Garra}, the authors studied a generalization of the fractional Poisson process in which the waiting times are independent but not identically distributed.
For a given sequence $J_n$, $n\geq 0$, of independent r.v.'s with distribution 
\begin{align} \label{Mittag leffler distribution state dependent}
P(J_k \geq t)= E_{\alpha _k}(-\lambda t^{\alpha _k}), \qquad \alpha _k \in (0,1),
\end{align}
they defined the state dependent fractional Poisson process as
\begin{align}
\mathcal{N}(t)=n \qquad  T_n    \leq t < T_{n+1} \label{definizione state dependent}
\end{align} 
where $T_n= \sum _{k=0}^{n-1} J_k$, $T_0=0$. Further they proved that the state probabilities $p_{k}(t):= P(\mathcal{N}(t)=k \mid \mathcal{N}(0) = 0)$ are such that
\begin{align} \label{probabilita di stato state dependent}
\widetilde{p}_k(s)= \int _0 ^{\infty}e^{-st} p_k(t)dt=   \frac{\lambda ^{k}s^{\alpha_k-1}}{\prod_{i=0}^k (s^{\alpha _i}+\lambda)}.
\end{align}
The authors noticed that apparently $\mathcal{N}(t)$ is not governed by fractional differential equations, since the state probabilities corresponding to \eqref{probabilita di stato state dependent} do not solve the fractional ``forward" equation with variable order derivative
\begin{align}
\mathcal{D}_t^{\alpha_k} p_k(t)=-\lambda p_k(t)+\lambda p_{k-1}(t).
\label{44}
\end{align}
Moreover, the construction of $\mathcal{N}(t)$ as a time-changed process was not clear.

An application of our results to this particular situation also sheds light to such problems. In particular, from Theorem \ref{teoback} it follows that the transition probabilities, which have explicit Laplace transform
\begin{align} \label{aaa}
\tilde{p}_{i,j}(s)= \int _0^\infty e^{-st} p_{i,j}(t)dt= \frac{s^{\alpha _j -1}\lambda ^{j-i}}{\prod _{k=i}^j (\lambda +s^{\alpha _k})},
\end{align}
are really related with fractional calculus and indeed they solve the following system of fractional Kolmogorov backward equations
\begin{align} \label{equazione backward Poisson state-dependent}
\mathcal{D}_t^{\alpha _i}  p_{i,j}(t)&=-\lambda p_{i,j}(t)+\lambda p_{i+1,j}(t) \qquad j\geq i\\
p_{i,j}(0)&= \delta _{i,j}. \notag
\end{align}
Moreover, by Theorem \ref{tetimechfrac} it follows that to obtain \eqref{definizione state dependent}, a standard Poisson process must be composed with a dependent multistable subordinator. In what follows we determine the structure of the forward equations for semi-Markov processes defined as in \eqref{processo da studiare}. The state-dependent fractional Poisson process is therefore a particular case, so a fractional forward equation for this process can be really written down and does not coincide with \eqref{44}.

We observe that in this framework the authors of \cite{Beghin3} also studied the time-change of a Poisson process by means of the inverse of a multistable subordinator. However in this case the multistable subordinator is assumed independent from the Poisson process. Hence the resulting process is a random walk with independent but non identically distributed sojourn times, whose probability law is a time-inhomogeneous generalization  of the Mittag-Leffler distribution, which is not coinciding with \eqref{Mittag leffler distribution state dependent}. Hence, the  process in \cite{Beghin3} is not a model for motions in heterogeneous media, but rather in an environment whose physical conditions change over time.
\end{ex}

We now derive the system of forward (or Fokker-Planck) equations governing the process \eqref{processo da studiare}. Since the random walk has place in a heterogeneous medium, an adequate kinetic description of the process requires variable order fractional operators.
The proof of the following theorem is based on the quantities
\begin{align*}
J_i^+(t)dt& =P((X_{t+dt}=i) \cap (X_t\neq i)) \notag \\
J_i^-(t)dt& = P((X_{t+dt}\neq i) \cap (X_t= i)) \notag 
\end{align*}
where $J_i^+$ and $J_i^-$ represent the gain and the loss fluxes for the state $i$.
It is intuitive that to deal with $J_i^+$ and $J_i^-$ it is convenient to assume that the when the process jumps (so when $t=T_n$ for some $n$) it can't jump in its current position. Hence we will assume in the following theorem that $h_{i,i}=0$. However this is not strictly necessary and by adapting the notations one can generalize and remove the assumption.
\begin{os}
\label{remrendens}
The Markov property is a consequence of the lack of memory property of the exponential distribution, which roughly states that when the Markov process is at $i$ the probability of having a jump in an infinitesimal interval of time $dt$ is $\lambda_idt$, so the rate $\lambda_i$ is constant in time. Of course since the lack of memory is not true for other distributions, we must have here that the rate varies with time, i.e., it is $\lambda_i u_i(t) dt$. It turns out that in this case the function $u_i(t)$ is given by
\begin{align}
u_i(t) \, = \, \frac{t^{\alpha_i-1}}{\Gamma(\alpha_i)}, \qquad \alpha_i \in (0,1).
\end{align}
This fact will be proved in the general situation in Section \ref{6}. We remark that the probability that in an infinitesimal interval there is more than one jump is $o(dt)$. This is because by construction we know that $X(t)$ is the same process as $M(t^\prime)$ with $t = \sigma^M(t^\prime)$: since $\sigma^M(t^\prime)$ is, a.s., strictly increasing on any finite interval of time and continuous, we have that within an infinitesimal time interval $dt$ also the process $X(t)$ performs at most one jump.
\end{os}
\begin{te}
\label{teforward}
Let $X$ be the process in \eqref{processo da studiare}. Assume that the $h_{i,i}=0$. Then the transition probabilities
\begin{align*}
p_{l,i}(t)\, = \,  P(X(t)=i|X(0)=l, \gamma (0)=0) \qquad l,i\in \mathcal{S}
\end{align*}
solve the following system of fractional forward equations
\begin{align}
\frac{d}{dt}p_{l,i}(t)= \sum _k g_{k,i}\,  ^R\mathcal{D}_t^{1-\alpha _k} p_{l,k}(t).\label{F}
\end{align}
\end{te}
\begin{proof}
The probability that the process performs more than one jump in an infinitesimal interval is $o(dt)$ is discussed in Remark \ref{remrendens}.
Let $J_i^+(t)dt$ be the probability of reaching the state $i$ during the time interval $[t,t+dt)$, i.e.,
\begin{align*}
J_i^+(t)dt& =P((X_{t+dt}=i) \cap (X_t\neq i))
\end{align*}
and let $J_i^-(t)dt$ be the probability of leaving the state $i$ during the time interval $[t,t+dt)$, i.e.,
\begin{align*}
J_i^-(t)dt& = P((X_{t+dt}\neq i) \cap (X_t= i)).
\end{align*}
Then we have under $P:=P \l \cdot \mid X(0)=l, \gamma (0) = 0 \r$
\begin{align}
&P(\{ X_{t+dt}=i\}) \notag \\ &= P(\{ X_{t+dt}=i\} \cap \{X_t=i\}) + P(\{ X_{t+dt}=i\} \cap \{X_t\neq i\})\notag \\
&= P(\{X_t=i\})-P(\{X_{t+dt } \neq i \} \cap \{X_t=i\})+P(\{ X_{t+dt}=i\} \cap \{X_t\neq i\})
\label{forw}
\end{align}
which, in our notations, reads
\begin{align*}
p_{l,i}(t+dt)= p_{l,i}(t)-J_i^-(t)dt+J_i^+(t)dt
\end{align*}
or, equivalently,
\begin{align}
\frac{d}{dt}p_{l,i}(t)=  J^+_i(t)-J^-_i(t), \label{A} \qquad t\geq 0.
\end{align}
By the total probability law, the ingoing flux can be computed as
\begin{align*}
J^+_i(t)= \sum _{k\neq i} J^-_k(t)h_{k,i}
\end{align*}
where $h_{k,i}$ is the matrix of the embedded chain. Then we obtain the following balance equation (expressing the conservation of probability mass)
\begin{align}
\frac{d}{dt}p_{l,i}(t)&=  \sum _{k\neq i} J^-_k(t)h_{k,i}-J^-_i(t) \notag 
\end{align}
which can also be written as
\begin{align}
&\frac{d}{dt}p_{l,i}(t)= \sum _k J^-_k(t) (h_{k,i}-\delta _{k,i}). \label{C}
\end{align}
The main goal is to compute the outgoing flux $J_i^-(t)$. It can be viewed as the sum of two contributions
\begin{align*}
J^-_i(t)dt= A_i^1(t)dt+A_i^2(t)dt
\end{align*}
where $A_i^1(t)dt$ is the probability to be initially in the state \textit{i} and to remain there for a time exactly equal to $t$ and 
$A_i^2(t)dt$ is the probability to reach the state \textit{i} at  time $t'<t$ and to remain there for a time $t-t'$. 
Thus
\begin{align*}
J^-_i(t)= f_i(t)p_{l,i}(0)+\int_0^t f_i(t-t')J^+_i(t')dt'
\end{align*}
where $p_{l,i}(0)= \delta _{l,i}$ and $f_i(t)$ the probability density of the holding time in $i$. We can eliminate $J^+_i(t)$ since by \eqref{A} we have $J^+_i(t)= \frac{d}{dt}p_{l,i}(t)+J^-_i(t)$. We thus obtain
\begin{align}
J^-_i(t)= f_i(t)p_{l,i}(0)+\int_0^t f_i(t-t')\bigl(\frac{d}{dt'}p_{l,i}(t')+J^-_i(t')    \bigr )dt' \label{B}
\end{align}
which is an integral equation in $J^-_i(t)$.
By applying the Laplace transform to (\ref{B}) we obtain
\begin{align*}
\tilde{J}^-_i(s)= \tilde{f}_i(s)p_{l,i}(0)+ \tilde{f}_i(s) \bigl (s\, \tilde{p}_{l,i}(s)-p_{l,i}(0)  + \tilde{J}^-_i(s)  \bigr)
\end{align*}
which gives
\begin{align}\label{E}
\tilde{J}^-_i(s)=\frac{s \tilde{f}_i(s)}{1-\tilde{f}_i(s)}\tilde{p}_{l,i}(s).
\end{align}
By assuming that the holding times follow a Mittag-Leffler distribution we have by \eqref{lapldensmittag} that
\begin{align*}
\tilde{f}_i(s)= \frac{\lambda_i}{\lambda_i+s^{\alpha_i}}
\end{align*}
and thus formula \eqref{E} becomes
\begin{align*}
\tilde{J}^-_i(s)= \lambda_is^{1-\alpha _i}\, \tilde{p}_{l,i}(s).
\end{align*}
By reminding the definition of Riemann-Liouville derivative \eqref{defriemann}, we have
\begin{align}
J^-_i(t)= \lambda_i \, ^R\mathcal{D}_t^{1-\alpha _i}p_{l,i}(t)
\label{416}
\end{align}
and \eqref{C}  reduces to
\begin{align}
\frac{d}{dt}p_{l,i}(t)&= \sum_k \lambda _k \, ^R\mathcal{D}_t ^{1-\alpha _k}p_{l,k}(t)(h_{k,i}-\delta _{k,i})\notag \\
&= \sum _k \, ^R\mathcal{D}_t^{1-\alpha _ k}p_{l,k}(t)g_{k,i}.
\end{align}
and the proof is complete.
\end{proof}
\begin{os}
We remark that the reason why such equation is called forward is that the operator on the right side acts on the ``forward" variable $i$ but leaves unchanged the backward variable $l$. Indeed, the equation is derived by conditioning on the event of the last jump (reaching the final state $i$) that may have occurred  in a narrow interval near $t$. This last aspect is particularly clear by looking at \eqref{forw}.
\end{os}

\begin{os}
If $X$ is a Markov process, then the sojourn times are exponentially distributed, i.e.
\begin{align*}
\tilde{f}_i(s)= \frac{\lambda_i}{\lambda_i+s}.
\end{align*}
Then (\ref{E}) reduces to
\begin{align*}
\tilde{J}^-_i(s)= \lambda _i \, \tilde{p}_{l,i}(s)
\end{align*}
namely
\begin{align}
J^-_i(t)= \lambda _i p_{l,i}(t).\label{z}
\end{align}
Then the balance equation (\ref{C}) reduces to the forward Kolmogorov equation
\begin{align*}
\frac{d}{dt}p_{l,i}(t)&= \sum _k J^-_k(t) (h_{k,i}-\delta _{k,i})\\
&= \sum _k \lambda _k(h_{k,i}-\delta _{k,i})p_{l,k}(t)\\
&= \sum _k p_{l,k}(t) g_{k,i}.
\end{align*}
From the physical point of view, the dynamics of the Markovian case and that of the CTRW with Mittag-Leffler waiting times present a wide difference.
Indeed, in the Markov case, the outgoing flux $J_i^-(t)$ from the state $i$ at time $t$ is proportional to the concentration of particles at the present time $t$ (see \eqref{z}). Instead, in the CTRW with power law waiting times,  the outgoing flux depends on the particles concentration at past times, according to a suitable weight kernel (see \eqref{416}). 
\end{os}

\begin{ex}[The state dependent fractional Poisson process (continued)]
To conclude the discussion of Example \ref{expoi1} we remark here that the forward equation for the state-dependent fractional Poisson process can be written down by using Theorem \ref{teforward}. We have indeed that the probabilities $p_{l,i}(t) := P \l \mathcal{N}(t) = i \mid \mathcal{N}(0) = l \r$ satisfy
\begin{align}
\frac{d}{dt}p_{l,i}(t) \, = \, \lambda \l \, ^R\mathcal{D}_t^{1-\alpha_{i-1}}  p_{l,i-1}(t) - \, ^R\mathcal{D}_t^{1-\alpha_{i}} p_{l,i}(t) \r.
\end{align}
\end{ex}

\section{Convergence to the variable order fractional diffusion}
\label{5}
Suppose that the state space $\mathcal{S}$ is embedded in $\mathbb{R}$. Hence our processes can be viewed as processes on $\mathbb{R}$, whose distribution is supported on $\mathcal{S}$. So in this section we consider a suitable scale limit of a semi-Markov process and we show that the one time distribution converges to the solution of forward heat equation on $\mathbb{R}$ with fractional variable order 
\begin{align}
\frac{\partial}{\partial t}p(x,y,t)= \frac{1}{2} \frac{\partial^2}{\partial y^2}\l \, ^R\mathcal{D}_t^{1-\alpha (y)} p(x,y,t)\r \label{G}.
\end{align} 
Such equation has been derived for the first time in \cite{Gorenflo} exactly in the study of anomalous diffusion in heterogeneous media. Hence our method provides a semi-Markov framework to this equation. The homogeneous case is represented by the time-fractional diffusion equation
\begin{align}
\mathcal{D}_t^{\alpha }p(x,y,t)= \frac{1}{2} \frac{\partial^2}{\partial x^2} p(x,y,t)
\label{classicheat}
\end{align}
which is well known in literature and already has a probabilistic interpretation (see Remark \ref{remctrw} below for some details). This equation is related with anomalous diffusion (non-Fickian diffusion), see, for example, \cite{hairer} for a recent application.

Let us assume that the process defined in \eqref{processo da studiare} is a 
symmetric CTRW with Mittag-Leffler waiting times and with transition probabilities
\begin{align}
h_{i,j} = \begin{cases} \frac{1}{2}, \qquad & j = i-1, i+1, \notag \\ 0, & j \neq i-1, i+1, \end{cases}
\end{align}
and $\lambda_i=\lambda$. Since $i,j$ are labels for points on the real line we can safely assume that the distance in $\mathbb{R}$ of two near points of $\mathcal{S}$ is constant and equal to $\epsilon$, i.e., $|i-j |=\epsilon$ for $j=i-1,i+1$ where $|\cdot|$ is the euclidean distance in $\mathbb{R}$. Hence looking at the process in $\mathbb{R}$ the walker performs jumps of size $\epsilon$. Then define
\begin{align}
\mathbb{R}^2 \ni (x,y) \mapsto  p(x,y, t) \, = \, \begin{cases} p_{x,y}(t), \qquad & (x, y) \in \mathcal{S} \times \mathcal{S}, \notag \\ 0, & \text{otherwise}.  \end{cases}
\end{align}
The forward equation \eqref{F} reduces to
\begin{align}
\frac{d}{dt}p_{l,i}(t)=
\frac{1}{2}\lambda \left[ \, ^R \mathcal{D}_t^{1-\alpha _{i-1}}  p_{l,i-1}(t)+ \, ^R \mathcal{D}_t^{1-\alpha _{i+1}} p_{l,i+1}(t)-2 \,^R \mathcal{D}_t^{1-\alpha _{i}} p_{l,i}(t)\right].
\label{forwdiff}
\end{align}
By considering now the auxiliary function
\begin{align}
u(x,y,t)= \, ^ R \mathcal{D}_t^{1-\alpha (y)}  p(x,y,t)
\end{align}
we can rewrite \eqref{forwdiff} as
\begin{align}
\frac{\partial}{\partial t}p(x,y,t)=\frac{1}{2}\lambda \l u(x,y-\epsilon,t)+u(x,y+\epsilon,t)-2u(x,y,t) \r.
\end{align}
By setting $\lambda = 1/\epsilon^2$ and letting $\epsilon \to 0$, the second derivative $\frac{\partial^2}{\partial y^2}u(x,y,t)$ arises. We thus obtain
\begin{align}
\frac{\partial}{\partial t}p(x,y,t)= \frac{1}{2} \frac{\partial^2}{\partial y^2}\l \, ^R\mathcal{D}_t^{1-\alpha (y)} p(x,y,t) \r.
\end{align} 
Note that the same scaling limit of a symmetic CTRW on a $d$-dimensional lattice leads to an analogous equation exibiting  the Laplace operator in place of the second order derivative.

Equation \eqref{G} can be obtained phenomenologically by combining the continuity equation
\begin{align}
\frac{\partial }{\partial t } p(x,y,t)= - \frac{\partial }{\partial y} q(x,y,t)
\end{align}
with an ad-hoc fractional Fick’s law regarding the flux $q(x,y,t)$:
\begin{align}
q(x,y,t)=-  \,  \frac{\partial}{\partial y} \, ^R\mathcal{D}_t^{1-\alpha (y)} p(x,y,t).
\end{align}
The fractional derivative in this expression provides a weighted average of the density gradient over the prior history, provided that the kernel of the average depends on the position $y$.

In terms of probability theory, the picture is completed by the backward heat equation with fractional variable order
\begin{align}
\mathcal{D}_t^{\alpha (x)}  p(x,y,t)= \frac{1}{2} \frac{d^2}{dx^2}p(x,y,t). \label{H}
\end{align}
Such equation has been derived in \cite{Orsingher2}, where the authors studied the convergence of the resolvent of semi-Markov evolution operators.
Heuristically, eq. $\eqref{H}$ can be obtained from the backward equation \eqref{aa} adapted to the case of a symmetric random walk with Mittag-Leffler waiting times:
\begin{align}
\mathcal{D}_t^{\alpha _i} p_{i,j}(t)= \frac{1}{2}\lambda \l p_{i+1,j}(t)+p_{i-1,j}(t)-2p_{i,j}(t)\r.
\end{align}
Indeed, passing to a lattice of size $\epsilon$ we have
\begin{align}
\mathcal{D}_t^{\alpha (x)}p(x,y,t) = \frac{1}{2}\lambda \bigl ( p(x+\epsilon,y,t)+p(x-\epsilon,y,t)-2p(x,y,t)
\bigr ).\end{align}
By assuming $\lambda= 1/ \epsilon ^2$, the limit $\epsilon \to 0$ gives the desired result.

\begin{os}
In our derivation of the fractional heat equations, the diffusion coefficient is put equal to 1.
However, the equations reported in the literature (e.g. \cite{Fetodov, Gorenflo}) usually exhibit a space-dependent diffusion coefficient.
To obtain this from a CTRW scheme, it is sufficient to assume a space dependent intensity $\lambda$ such that, in the limit of small $\epsilon$, it is of order  
$1/\epsilon ^2$, that is $\lambda (x)= \widetilde{\lambda}(x)/ \epsilon ^2$. Then, it is natural to define the diffusion coefficient as
\begin{align}
k(x)= \widetilde{\lambda}(x)
\end{align}
and to repeat the same scaling limit argument in order to obtain both the forward equation
\begin{align}
\frac{\partial}{\partial t}p(x,y,t) \, = \, \frac{1}{2} \frac{\partial^2}{\partial y^2}\l k(y) ^R\mathcal{D}_t^{1-\alpha (y)}  p(x,y,t)\r
\end{align}
and the backward equation
\begin{align}
\mathcal{D}_t^{\alpha (x)}p(x,y,t)= \frac{1}{2} k(x) \frac{\partial ^2}{\partial x^2}   p(x,y,t).
\end{align}
\end{os}

\begin{os}
The mean square displacement of a subdiffusion in a homogeneous medium  grows slower with respect to the Brownian motion, i.e., $ \overline{x^2}(t) \sim t^{{\alpha}}$, where $\alpha \in (0,1)$. 
Such a  process can be represented as a Brownian motion delayed by an independent inverse stable subordinator.
For a study of its long time asymptotic properties, consult \cite{Shilling}. Concerning subdiffusion in heterogeneous media described by our equation, the picture is much more complicated and some unexpected phenomena arise. For example, in \cite{Fetodov}, the authors find that in the long time limit the  CTRW process is localized at the lattice point where $\alpha (x)$ 
has its minimum, a phenomenon called ``anomalous aggregation". This suggests that the process does not enjoy the same ergodic properties reported in \cite{Shilling} for subdiffusion in homogeneous media.
\end{os}

\begin{os}
The transition probability $p(x,y,t)$ is the fundamental solution to the partial differential equation (\ref{H}), and thus it is interesting to consider some well-posedness issues.  The most recent result in this direction can be found in \cite{Kian}
where the authors consider the following Cauchy problem
\begin{align}
\begin{cases}
\rho(x) \mathcal{D}_t^{\alpha (x)}u(x,t)- \Delta u(x,t)= f(x,t) \qquad   x\in \Omega , t \in (0,T)\\
u(x,0)=u_0(x)\\
u(x,t)= 0 \qquad x \in \partial \Omega , t\in (0,T)
\end{cases}
\end{align}
and prove that, under suitable assumptions on the source term $f$ and the initial datum $u_0$,  there exists a unique weak solution $u(x,t)$ in the sense of \cite[Thm 2.3]{Kian}.
\end{os}

\begin{os}
\label{remctrw}
There are several results on CTRWs limit processes which can be applied in this situation by making some further assumptions \cite{marcincoupled, meertri, Meerschaert2, meerstra, strakahenry}. We discuss here an example. Consider the couple process $\l X_n, T_n \r$ and introduce a scale parameter $c$, so the process is $\l X_n^c, T_n^c \r$. By making assumptions on the weak convergence of probability measures of the process $\l X_{[u/c]}^c, T_{[u/c]}^c \r \to \l A(t), D(t) \r$ as $c \to 0$ (e.g. as in \cite[Theorem 3.6]{strakahenry}) one has that
\begin{align}
X^c(t) \, \to \, X^0 (t):=  A(E(t-)) 
\end{align}
where $E(t)$ is the hitting-time of $D$. In our situation the processes are non independent and hence ``coupled" in the language CTRW. 
Of course when $A$ is a Brownian motion and $D(t)$ is an independent $\alpha$-stable subordinator we are in the equivalent homogeneous situation of this section: the one time distribution of $A(E(t))$ solves indeed eq \eqref{classicheat} in which $\alpha$ is constant \cite{fracCauchy}. Here we can conjecture that in order to obtain a process governed by \eqref{G} we must assume that $A$ is still a Brownian motion and that $D(t)$ is a multistable subordinator $\sigma(t)$ obtained as a limit case of the piecewise stable subordinator of Definition \eqref{defpiec}. So we argue that $D(t)$ must be a multistable subordinator (dependent on $A(t)$) whose L\'evy measure is the limit of the L\'evy measure of a piecewise stable subordinator, i.e., conditionally on a Brownian path $A(t, \omega)$
\begin{align}
d\nu(ds, t)/ds \, = \, \int_\mathbb{R} \frac{\alpha(x)s^{-\alpha(x)-1}}{\Gamma(1-\alpha(x))} \, \mathds{1}_{\ll A(t, \omega) = x \rr} dx.
\label{multilim}
\end{align}
Then one can define $E(t)$ as the hitting-time of $\sigma(t)$. This require several further investigations.
\end{os}

\section{Arbitrary holding times and integro-differential Volterra equations}
\label{6}
The construction of Theorem \ref{tetimechfrac} is based on the notion of multistable subordinator. In \cite{Orsingher1} the authors introduced the more general class of inhomogeneous subordinators, i.e., non decreasing processes with independent and non stationary increments. By using these, it is possible to define a new type of CTRWs, which is constructed in the same way as \eqref{processo da studiare}, except for the distributions of the waiting times, which are no more Mittag-Leffler. 

Indeed, for any $i \in \mathcal{S}$, consider a L\'evy measure $\nu (dx, i)$ which defines a homogeneous subordinator $\sigma ^i$ such that
\begin{align*}
\mathbb{E} e^{-s\sigma ^i (t)}= e^{-t f(s,i)}
\end{align*}
where 
\begin{align}
f(s,i) \, = \, \int_0^\infty \l 1-e^{-sw} \r \nu(dw,i)
\end{align}
is the Laplace exponent of $\sigma ^i$.
Let $L^i(t) = \inf \ll \tau : \sigma^i (\tau)>t \rr$ be the right continuous hitting time of $\sigma ^i$. For any $i \in \mathcal{S}$ we assume $\nu ((0, \infty), i)= \infty$, in such a way that $\sigma ^i$ is a.s. stricly increasing, $L^i$ has a.e. continuous sample paths, and, for any $t>0$, $\sigma ^i(t)$ and $L^i(t)$ are absolutely continuous random variables. We are now ready to define the following CTRW:
\begin{align}
X(t)= X_n \qquad        T_n \leq t< T_{n+1} ,\label{ultimo processo}
\end{align}
where $T_n= \sum _{k=0}^{n-1}J_k$, $T_0=0$, and
\begin{align}
P \l J_n>t \mid X_n=i \r \, = \, \overline{F}_i(t)= \mathbb{E}e^{-\lambda_i L^i(t)}.
\label{63}
\end{align}
The generalization of Theorem \eqref{tetimechfrac} is immediate. Let  $M$ be a Markov process defined as in \eqref{definizione processo markov}. Moreover, let $\sigma^M(t)$ be an inhomogeneous subordinator dependent on $M$ {whose L\'evy measure, conditionally on $V_1 =v_1,V_2=v_2, \cdots$ and $X_1=x_1, X_2=x_2, \cdots$ is given by
\begin{align} 
\nu ^M(dx,t) = \nu (dx,i),  \qquad     v_{i} \leq t< v_{i+1}.
\end{align}
Let $L^M(t)$ be the right continuous inverse of $\sigma^M(t)$.
Then the time-changed process $M(L^M(t))$ is the same process as \eqref{ultimo processo}. To prove this, the key point is the fact that
\begin{align}
\int_0^\infty e^{-s t} P \l J_n > t \mid X_n = i \r \, = \, \frac{f(s, i)}{s} \frac{1}{\lambda_i+f(s,i)},
\end{align}
namely, conditionally to $X_n=i$, $J_n$ has a density $\psi _i$ with Laplace transform
\begin{align}
 \int_0^\infty e^{-s \tau} \psi _i(\tau) d\tau\, = \, \frac{\lambda_i}{\lambda_i + f(s,i)},
 \label{322}
\end{align}
which is coinciding with 
$\mathds{E}(e^{-s \sigma ^i (E_n)}|X_n=i) $.

\subsection{ Integro-differential Volterra equations with position dependent kernel}
To obtain a backward equation, we resort again to \eqref{Markov renewal equation} and applying Laplace transform to both sides  yields
\begin{align} 
\widetilde{p}_{i,j}(s)= \frac{f(s,i) }{s(\lambda _i+f(s,i) )}\delta _{i,j} +\sum _l h_{i,l}\, \widetilde{p}_{l,j}(s)\frac{\lambda _i}{\lambda _i+ f(s,i)},
\end{align}
which can be rearranged as
\begin{align}
f(s,i)\widetilde{p}_{i,j}(s) -s^{-1}f(s,i) \delta_{i,j} \, = \, \sum _k g_{i,k}\widetilde{p}_{k,j}(s)
\label{laplgenerale}
\end{align}
where again $g_{i,j}=\lambda _i (h_{i,j}-\delta _{i,j})$. Inverting Laplace transform in \eqref{laplgenerale} however does not yield to a time-fractional equation. By using indeed \cite[Lemma 2.5 and Proposition 2.7]{toaldo} we get the inverse Laplace transform
\begin{align}
\frac{d}{dt} \int_0^t p_{i,j}(t^\prime) \, \bar{\nu}(t-t^\prime,i) \, dt^\prime \, - \delta_{i,j} \bar{\nu}(t,i) \,  = \, \sum_k g_{i,k}p_{k,j}(t)
\label{oltgen}
\end{align}
where $\bar{\nu}(t,i):=\nu((t, \infty),i)$, provided that the integral function is differentiable. It is clear that in the situation of Theorem \ref{tetimechfrac} one has
\begin{align}
\bar{\nu}(t,i) \, = \, \int_t^\infty \frac{\alpha_i w^{-\alpha_i-1}}{\Gamma (1-\alpha_i)} dw \, = \, \frac{t^{-\alpha_i}}{\Gamma(1-\alpha_i)}
\end{align}
and the operator on the left-hand side of \eqref{oltgen} becomes a fractional Caputo derivative.

We now also derive a forward equation. Let 
\begin{align}
N^*(t)= \max \{n: T_n\leq t \}
\end{align}
be the number of renewals for the process \eqref{ultimo processo} up to time $t$. Of course, conditionally to $X_1=x_1$, $X_2=x_2, \cdots,$ we have that $N^*$ is a birth process (with rates $\lambda _{x_1}$, $\lambda _{x_2}, \cdots$) time changed by the dependent time process $L^{M}$. 
Our attention focuses on the quantity, conditionally on $\ll X(t)=i \rr$,
\begin{align}
\lim _{\Delta t \to 0} \frac{\mathbb{E} [N^*(t+\Delta t) ]- \mathbb{E}[N^*(t)]}{\Delta t} \label{renewal density}
\end{align}
which we call renewal density (in the spirit of \cite[page 26]{Cox}) and specifies the mean number of renewals to be expected in a narrow interval near $t$ conditionally on the current position.
Since we condition on $X(t)=i$, \eqref{renewal density} is obviously depending on $i$ and $N^*(t+dt)-N^*(t)$ behaves like $N(L^i(t+dt))-N(L^i(t))$ where $N$ is a standard birth process.
Thus the limit \eqref{renewal density} can be computed as
\begin{align}
m_i(t)= \, & \, \frac{d}{dt}\mathbb{E}\left[N^*(t)\right]= \frac{d}{dt}\mathbb{E}N(L^i(t))= \frac{d}{dt}\lambda _i \mathbb{E}L^i(t)\notag \\
= \, & \lambda _i \frac{d}{dt} \int _0^\infty P(L^i(t)>w)dw\, = \,  \lambda _i  \frac{d}{dt}\int _0^\infty P(\sigma ^i (w)<t)dw.
\label{renfun}
\end{align}
The function
\begin{align}
t \mapsto u^i(t):=\frac{d}{dt}\int _0^\infty P(\sigma ^i (w)<t)dw
\label{rendens}
\end{align}
on the right-hand side of \eqref{renfun} is said to be, in the language of potential theory (e.g. \cite{pottheory}), the potential density of the subordinator $\sigma^i$ and is such that (e.g. \cite[Section 1.3]{bertoins})
\begin{align}
\int_0^\infty e^{-st} u^i(t) dt\, = \, \frac{1}{ f(s, i)}
\end{align}
provided that the derivative in \eqref{rendens} exists a.e. Heuristically $u^i(t)$ represents the mean of the total amount of time spent by the subordinator $\sigma ^i$ in the state $dt$.

To obtain a forward equation we can follow the same line of section \ref{4} up to formula \eqref{E}. Then,  by using \eqref{322}, the outgoing flux has Laplace transform
\begin{align*}
\tilde{J}_i^-(s)= \lambda _i \frac{s}{f(s,i)}\tilde{p}_{l,i}(s)
\end{align*}
and thus the convolution theorem gives
\begin{align}
J^-_i(t)&=  \frac{d}{dt} \int_0^t p_{l,i}(\tau) \, m_i(t-\tau) d\tau   \\& =  \lambda_i \frac{d}{dt} \int_0^t p_{l,i}(\tau) \, u_i(t-\tau) d\tau.
\label{zz}
\end{align}
Finally, \eqref{C}  reduces to
\begin{align} \label{forward finale}
\frac{d}{dt}p_{l,i}(t)= \sum _k g_{k,i} \frac{d}{dt} \int_0^t p_{l,k}(s) \, u_k(t-s) ds ,
\end{align}
which is the forward equation for our process. It is straightforward to prove that in the fractional case the renewal density relatively to the state $k$ reads
\begin{align*}
m_k(t)=\lambda_k \frac{t^{\alpha_k-1}}{\Gamma (\alpha_k)}
\end{align*}
and the operator on the right-hand side of \eqref{forward finale} reduces to the Riemann Liouville derivative $^R\mathcal{D}^{1-\alpha_k}$.
With the above discussion we have proved the following result.
\begin{te}
Let $X(t)$ be a process like \eqref{ultimo processo} with holding times $\overline{F}_i(t)$ given by \eqref{63}. Further assume that $m_i(t)$ exists for any $i$. Then the probabilities $p_{i,j}(t)$ satisfy the backward equation \eqref{oltgen} as well as the forward equation \eqref{forward finale}.
\end{te}

\end{document}